\documentclass[a4paper]{amsart}
\usepackage{amsmath,amssymb,amsthm,amscd}
\usepackage[all]{xy}

\title[TCMC, idempotent ideals, and the transfinite radical]{Telescope conjecture, idempotent ideals, and the transfinite radical}
\date{\today}

\author{Jan \v S\v tov\'\i\v cek}
\address{Institutt for matematiske fag, Norges teknisk-naturvitenskapelige universitet \\ 
N-7491 Trondheim, Norway}
\email{stovicek@math.ntnu.no}

\subjclass[2000]{18E35 (primary), 16E30, 16G60 (secondary)}
\keywords{Telescope conjecture, idempotent ideals, transfinite radical}

\thanks{The author was supported by the Research Council of Norway through Storforsk-project Homological and geometric methods in algebra, and also by the grant GA\v CR 201/05/H005 and the research project MSM 0021620839.}

\renewcommand{\iff}{if and only if }
\newcommand{\st}{such that }
\newcommand{\Z}{\mathbb{Z}}
\newcommand{\Q}{\mathbb{Q}}
\newcommand{\Modl}{\mathrm{Mod}\,}
\newcommand{\stModl}{\mathrm{\underline{Mod}}\,}
\newcommand{\modl}{\mathrm{mod}\,}
\newcommand{\modmodl}[1]{{\mathrm{fp}(\mathrm{mod}\,{#1}, \Ab)}}
\newcommand{\stmodl}{\mathrm{\underline{mod}}\,}
\newcommand{\ModL}{{\mathrm{Mod}\,\Lambda}}
\newcommand{\stModL}{\mathrm{\underline{Mod}}\,\Lambda}
\newcommand{\modL}{{\mathrm{mod}\,\Lambda}}
\newcommand{\modmodL}{{\mathrm{fp}(\mathrm{mod}\,\Lambda, \Ab)}}
\DeclareMathOperator{\Hom}{Hom}
\DeclareMathOperator{\End}{End}
\DeclareMathOperator{\Ext}{Ext}

\DeclareMathOperator{\Ker}{Ker}
\DeclareMathOperator{\Img}{Im}

\DeclareMathOperator{\rad}{rad}
\newcommand{\radL}{\rad_\Lambda}

\newcommand{\spanned}[1]{\langle #1 \rangle}

\DeclareMathOperator{\KGdim}{KGdim}
\newcommand{\A}{\mathcal{A}}
\newcommand{\B}{\mathcal{B}}
\newcommand{\C}{\mathcal{C}}

\newcommand{\clS}{\mathcal{S}}
\newcommand{\T}{\mathcal{T}}

\newcommand{\X}{\mathcal{X}}
\newcommand{\add}{\mathrm{add}}

\newcommand{\Ab}{\mathrm{Ab}}

\theoremstyle{plain}
\newtheorem{thm}{Theorem}
\newtheorem{prop}[thm]{Proposition}
\newtheorem{lem}[thm]{Lemma}
\newtheorem{cor}[thm]{Corollary}
\newtheorem*{conj}{Conjecture}
\theoremstyle{definition}
\newtheorem{defn}[thm]{Definition}
\theoremstyle{remark}
\newtheorem*{rem}{Remark}

\newtheorem{expl}[thm]{Example}

\begin{document}
\begin{abstract}
We show that for an artin algebra $\Lambda$, the telescope conjecture for module categories is equivalent to certain idempotent ideals of $\modL$ being generated by identity morphisms. As a consequence, we prove the conjecture for domestic standard selfinjective algebras and domestic special biserial algebras. We achieve this by showing that in any Krull-Schmidt category with local d.c.c.\ on ideals, any idempotent ideal is generated by identity maps and maps from the transfinite radical.

\end{abstract}
\maketitle

\section*{Introduction}

The aim of this paper is to further develop and apply connections between seemingly rather different topics in algebra:
\begin{enumerate}
\item localizations of triangulated compactly generated categories;
\item theory of cotorsion pairs and induced approximations;
\item the structure of idempotent ideals in a module category;
\item representation type of a finite dimensional algebra.
\end{enumerate}

The main motivation for this paper was point $(1)$, the study of so called smashing localizations in triangulated compactly generated categories. There is an important conjecture, the telescope conjecture, which roughly says that any smashing localization of a compactly generated triangulated category comes from a set of compact objects. For an extensive study of this problem and explanation of the terminology we refer to work by Krause~\cite{K1,K2}. Even though the conjecture is known to be false in this generality---see \cite{Kel} for a simple algebraic counterexample---it is not resolved for many particular important settings. Such special solutions would still have significant consequences. In the case of unbounded derived categories of rings, this is discussed in~\cite{K2}.

In this paper, we will focus on another setting. Let $R$ be a quasi-Frobenius ring (that is, the projective and injective left modules coincide), and $\stModl{R}$ be the stable module category of left $R$-modules. Then $\stModl{R}$ is a triangulated compactly generated category in the sense of~\cite{K1,K2}. If, moreover, $R$ is a self-injective artin algebra, the telescope conjecture has been translated by Krause and Solberg~\cite{KrS} to a statement about modules, or more precisely about certain cotorsion pairs of modules. The precise statements and explanation of terminology are given below. Recently, a positive solution to the telescope conjecture for stable module categories over finite group algebras was annouced by the authors of~\cite{BIK}. Their methods are, however, closely tied to group algebras and do not allow direct generalization to other self-injective artin algebras. We will develop an alternative approach.

The above mentioned version of the telescope conjecture for cotorsion pairs of modules from~\cite[\S 7]{KrS} makes sense not only for self-injective artin algebras, but in fact for any associative ring with unit, leading to a problem in homological algebra which is of interest by itself (cf.~\cite{AST,SS}). Even though one loses the translation to triangulated categories, similarities between the new and the original settings are striking and have been analyzed more in detail in~\cite{SS}.

In the present paper, we further develop the approach from~\cite{SS} and show that the telescope conjecture for module categories depends on the structure of certain idempotent ideals of the category of finitely presented modules. This is another analogy to so called exact ideals from~\cite{K2}. Further, we prove that the structure of idempotent ideals in the category of finitely presented modules over an artin algebra, as well as in many other categories studied by representation theory, heavily depends on idempotent ideals inside the radical. In particular, if there are no non-zero idempotent ideals in the radical, we get a positive answer to the telescope conjecture.

The condition of no non-zero idempotent ideals in the radical of the module category seems to be closely related to the domestic representation type. These notions were proved to coincide for special biserial algebras by Schr\"oer~\cite{S,PS}. A stronger but closely related condition when the infinite radical is nilpotent was studied by several authors, see for example~\cite{KeS,Sk,CMMS1,CMMS2}. Our main interest in the existing results stems from the fact that they provide us with non-trivial examples of artin algebras over which the telescope conjecture for module categories holds. Some of them, coming from a paper by Skowro\'nski and Kerner~\cite{KeS}, are self-injective, thus allowing us to go all the way back and get a statement about smashing localizations of their stable module categories.

Another condition which seems to be closely related to both the domestic representation type and vanishing of the transfinite radical is that of the Krull-Gabriel dimension of an artin algebra being an ordinal number. The concept of the Krull-Gabriel dimension of a ring $R$ can be interpreted as a measure for complexity of both the category $\modmodl R$ of finitely presented additive functors $\modl R \to \Ab$, and the lattice of primitive positive formulas over $R$. Using a result from~\cite{K3}, we prove that the telescope conjecture for module categories holds true if the Krull-Gabriel dimension of the artin algebra in question is an ordinal number.

\medskip

The author would like to thank \O{}yvind Solberg for several helpful discussions. The author is also grateful to Otto Kerner for his comments on idempotent ideals in the radical of some module categories and communicating the unpublished result by Dieter Vossieck mentioned in Section~\ref{sec:rad_tr}.

\section{Preliminaries}
\label{sec:prelim}

In this text, $\Lambda$ will always be an artin algebra and all modules will be left $\Lambda$-modules. Let us denote by $\ModL$ the category of all modules and by $\modL$ the full subcategory of finitely generated modules. Some results in this paper will be proved for more general categories: Krull-Schmidt categories with local d.c.c.\ on ideals as defined in Section~\ref{sec:idemp}. This setting includes $\modL$, derived bounded categories, categories of coherent sheaves, and other categories of representation theoretic significance. A reader who is not interested in the full generality can, nevertheless, read the corresponding statements as if they were stated for $\modL$.

A \emph{cotorsion pair} in $\ModL$ is a pair $(\A,\B)$ of full subcategories of $\ModL$ \st $\A = \Ker\Ext^1_\Lambda(-,\B)$ and $\B = \Ker\Ext^1_\Lambda(\A,-)$. A cotorsion pair is called \emph{hereditary} if in addition $\Ext^i_\Lambda(\A,\B)=0$ for all $i \ge 2$. This paper deals with the telescope conjecture for module categories (TCMC) as formulated in~\cite[Conjecture 7.9]{KrS}. Actually, we slightly alter the assumptions---we require the cotorsion pair in question to be hereditary (since the cotorsion pairs of interest in~\cite{KrS} always are) and relax the condition that~\cite{KrS} imposes on the class $\A$ of the cotorsion pair. We state the conjecture as follows:

\begin{conj}[A] \label{conj:A}
Let $\Lambda$ be an artin algebra and let $(\A,\B)$ be a hereditary cotorsion pair in $\ModL$ \st $\B$ is closed under taking filtered colimits. Then every module in $\A$ is a colimit of a filtered system of finitely generated modules from $\A$.
\end{conj}

Note that, in view of~\cite[Theorem 1.5]{AR}, we can equivalently replace filtered colimits by direct limits in the statement above. We say that a cotorsion pair $(\A,\B)$ in $\ModL$ is of \emph{finite type} if $\B = \Ker\Ext^1_\Lambda(\clS,-)$ for a set $\clS$ of finitely generated modules. Similarly, we define $(\A,\B)$ to be of \emph{countable type} if we can take $\clS$ to be a set of countably generated modules. With this definition we can for any particular algebra $\Lambda$ equivalently restate Conjecture (A) as follows, see~\cite[Corollary 4.6]{AST}:

\begin{conj}[B] \label{conj:B}
Let $\Lambda$ be an artin algebra and let $(\A,\B)$ be a hereditary cotorsion pair in $\ModL$ \st $\B$ is closed under taking direct limits. Then $(\A,\B)$ is of finite type.
\end{conj}

As a tool to handle the conjectures, we will need the notion of an ideal of an additive category. Let $\C$ be a skeletally small additive category. A class $\mathfrak I$ of morphisms in $\mathcal C$ is called a (2-sided) \emph{ideal} of $\C$ if $\mathfrak I$ contains all zero morphisms, and it is closed under addition and under composition with arbitrary morphisms from left and right, whenever the operations are defined. Let us denote $\mathfrak I(X,Y) = \mathfrak I \cap \Hom_\C(X,Y)$. Note that if $\C = \modL$ then $\mathfrak I(X,Y)$ is always a $k$-submodule of $\Hom_\Lambda(X,Y)$ where $k$ is the centre of $\Lambda$. Since $\C$ was assumed to be skeletally small, ideals of $\C$ form a set.

We say that an additive category $\C$ is a \emph{Krull-Schmidt category} if it is skeletally small, every indecomposable object of $\C$ has a local endomorphism ring, and every object of $\C$ (uniquely) decomposes as a finite coproduct of indecomposables. As an example to keep in mind, we can put $\C = \modL$. For Krull-Schmidt categories there is a prominent ideal called the \emph{radical}---it is the ideal generated by all non-invertible morphisms between indecomposable objects. We denote this ideal by $\rad_\C$ and if $\C = \modL$ we use the abbreviated notation $\rad_\Lambda$. Let us recall the well known fact that $\rad_\C$ contains no identity morphisms and, clearly, it is the maximal ideal with this property. Here and also later in this paper we, of course, mean no identity morphisms of non-zero objects since zero morphisms are in any ideal by definition.

Following an idea in~\cite{P}, we can inductively define transfinite powers $\mathfrak I^\alpha$ for any ideal $\mathfrak I$ and any ordinal number $\alpha$. Let $\mathfrak I^0$ be the ideal of all morphisms in $\C$ and $\mathfrak I^1 = \mathfrak I$. For a natural number $n \ge 1$, we define $\mathfrak I^n$ as usual as the ideal generated by all compositions of $n$-tuples of morphisms from $\mathfrak I$. If $\alpha$ is a limit ordinal, we define $\mathfrak I^\alpha = \bigcap_{\beta<\alpha} \mathfrak I^\beta$. If $\alpha$ is infinite non-limit, then uniquely $\alpha = \beta+n$ for some limit ordinal $\beta$ and natural number $n \ge 1$, and we set $\mathfrak I^\alpha = (\mathfrak I^\beta)^{n+1}$. Note that since we assume that $\C$ is skeletally small, the decreasing chain
$$
\mathfrak I^0 \supseteq \mathfrak I^1 \supseteq \mathfrak I^2 \supseteq \dots
\supseteq \mathfrak I^\alpha \supseteq \mathfrak I^{\alpha+1} \supseteq \dots
$$
stabilizes for cardinality reasons. Let us denote $\mathfrak I^* = \bigcap_\alpha \mathfrak I^\alpha$, the minimum of the chain.

We will focus mostly on the case when $\mathfrak I = \rad_\C$. In this case we call $\rad_\C^*$ the \emph{transfinite radical} of $\C$. Notice that not necessarily $\rad_\C^* = 0$, even when $\C = \modL$ for an artin algebra $\Lambda$---see the next section or~\cite{P,S}. The main goal of this paper is to prove that TCMC formulated as Conjecture (B) holds true over those artin algebras for which $\rad_\Lambda^* = 0$. This applies in particular to:
\begin{itemize}
\item \cite{KeS} standard selfinjective algebras of domestic representation type;
\item \cite{S} special biserial algebras of domestic representation type.
\end{itemize}
Recall that a finite dimensional algebra over an algebraically closed field is of \emph{domestic representation type} if there is a natural number $N$ \st for each dimension $d$, all but finitely many indecomposable modules of dimension $d$ belong to at most $N$ one-parametric families.

\section{Transfinite radical}
\label{sec:rad_tr}

Let $\C$ be an additive category. We call an ideal $\mathfrak I$ of $\C$ \emph{idempotent} if $\mathfrak I = \mathfrak I^2$. Equivalently, $\mathfrak I$ is idempotent \iff for each $f \in \mathfrak I$ there are $g,h \in \mathfrak I$ \st $f=gh$. Using idempotency, we can give the following characterization of the transfinite radical:

\begin{lem} \label{lem:rad_tr_char}
Let $\C$ be a Krull-Schmidt category. Then $\rad_\C^*$ is the unique maximal idempotent ideal of $\C$ which does not contain any identity morphisms.
\end{lem}

\begin{proof}
We use the same (just more verbose) proof as the one given for~\cite[8.10]{K3} for module categories.
Clearly, $\rad_\C^*$ contains no identity morphisms since neither $\rad_\C$ does. It is easy to check that $\rad_\C^*$ is idempotent \cite[Proposition 0.6]{P}. On the other hand, if $\mathfrak I$ is idempotent without identity maps, then $\mathfrak I = \mathfrak I^* \subseteq \rad_\C^*$ (since $\mathfrak I = \mathfrak I^\alpha$ for any ordinal $\alpha$ by idempotency). Hence $\rad_\C$ is maximal with respect to those two properties.
\end{proof}

There is also a useful characterization of the morphisms in $\rad_\C^*$ ``from inside'', sheding more light on the concept than a little cryptic definition as the intersection of a series of transfinite powers. The following statement has been proved in~\cite{P} for $\C=\modL$ using standard means similar to those when one deals with Krull dimension of a poset, and the proof reads equally well for any skeletally small Krull-Schmidt category:

\begin{lem} \label{lem:rad_tr_maps} \cite[Proposition 0.6]{P}
Let $\C$ be a Krull-Schmidt category and $f$ be a morphism in $\C$. Then $f \in \rad_\C^*$ \iff there exists a collection of morphisms $f_{pr}: X_r \to X_p$ in $\rad_\C$, one for each pair of rational numbers $p,r$ \st $0 \leq p < r \leq 1$, \st
\begin{enumerate}
\item $f_{ps} = f_{pr} f_{rs}$ whenever $p < r < s$;
\item $f_{01} = f$.
\end{enumerate}
\end{lem}

Note that the collection $(f_{pr})_{0\leq p<r\leq 1}$ is nothing else than an inverse system indexed by $[0,1] \cap \Q$. Using the two lemmas above, we can give some examples of what the transfinite radical can be:

\begin{itemize}
\item If $\Lambda$ is an artin algebra of finite representation type, then $\rad_\Lambda$ is nilpotent. Hence $\rad_\Lambda^* = 0$.

\item If $\Lambda$ is a tame hereditary artin algebra, then $\rad_\Lambda^{\omega+2} = (\rad_\Lambda^\omega)^3 = 0$. Hence $\rad_\Lambda^* = 0$.

\item If $\Lambda$ is a standard (that is, having a simply connected Galois covering) selfinjective algebra of domestic representation type, then $\rad_\Lambda^\omega$ is nilpotent \cite{KeS}. Hence $\rad_\Lambda^* = 0$.

\item If $\Lambda$ is a special biserial algebra, then $\rad_\Lambda^* = 0$ \iff $\rad_\Lambda^{\omega^2} = 0$ \iff $\Lambda$ is of domestic representation type. If $\Lambda$ is not domestic, then there exists an indecomposable $\Lambda$-module $X$ \st $0 \ne \radL^*(X,X) \subseteq \End_\Lambda(X)$ (see~\cite[Theorem 2 and Prop. 6.2]{S}).

\item As special case of the previous point, one may consider ``Gelfand-Ponomarev'' algebras $\Lambda_{m,n} = k[x,y]/(xy,yx,x^m,y^n)$, see~\cite{GP}. The algebra $\Lambda_{2,3}$ is not of domestic represetation type and provides a very illustrative example of non-zero maps in the transfinite radical, see~\cite{P}.

\item If $\Lambda$ is a wild hereditary artin algebra, it is conjectured that $\rad_\Lambda^\omega$ is idempotent. In view of Lemma~\ref{lem:rad_tr_char}, this cojecture can be rephrased as $\rad_\Lambda^* = \rad_\Lambda^\omega$.

\item It is an unpublished result due to Dieter Vossieck that for the category $\C = \modl k\langle x,y\rangle$ of finite dimensional modules over the free algebra $k\langle x,y\rangle$, the radical $\rad_\C$ is idempotent. In particular $\rad_\C^* = \rad_\C$.
\end{itemize}

There is an important consequence of some of the examples above for wild artin algebras over an algebraically closed field. Namely, they \emph{always} have the transfinite radical non-zero. Let us state this precisely.


\begin{defn} \label{defn:wild}
Let $\Lambda$ and $\Gamma$ be finite dimensional algebras over a field $k$ and let $F: \modl\Gamma \to \modL$ be an additive functor. Then $F$ is called a \emph{representation embedding} if $F$ is faithful, exact, preserves indecomposability (i.e.\ if $X$ is indecomposable, so is $FX$) and reflects isomorphism classes (i.e.\ if $FX \cong FY$ then also $X \cong Y$).

A finite dimensional $k$-algebra is called \emph{wild} if for any other finite dimensional algebra $\Gamma$ over $k$, there is a representation embedding $\modl\Gamma \to \modL$.
\end{defn}

The following statement immediately follows from~\cite[Proposition 6.2]{S} and~\cite[Lemma 0.2]{P} (the same idea is also presented in~\cite[8.15]{K3}):

\begin{prop} \label{prop:wild}
Let $\Lambda$ be a wild algebra over an algebraically closed field. Then $\rad_\Lambda^* \ne 0$. Moreover, there exists an indecomposable $\Lambda$-module $X$ \st $0 \ne \radL^*(X,X) \subseteq \End_\Lambda(X)$.
\end{prop}

\section{Idempotent ideals in Krull-Schmidt categories}
\label{sec:idemp}

Let $\mathfrak I$ be an ideal of a Krull-Schmidt category. Then clearly, if $\mathfrak I$ is generated by a collection of identity morphisms, it is necessarily an idempotent ideal. In the sequel we will show that in ``nice'' categories, any idempotent ideal is generated by a collection of identity morphisms together with some morphisms from the transfinite radical. To make the word \emph{nice} precise, we need the following definition:

\begin{defn} \label{defn:dcc}
A skeletally small additive category $\C$ is said to have \emph{local descending chain condition on ideals} if for any decreasing series
$$ \mathfrak I_0 \supseteq \mathfrak I_1 \supseteq \mathfrak I_2 \supseteq \dots $$
of ideals of $\C$ and any pair of objects $X,Y$ in $\mathcal C$, the decreasing chain
$$ \mathfrak I_0(X,Y) \supseteq \mathfrak I_1(X,Y) \supseteq \mathfrak I_2(X,Y) \supseteq \dots $$
stabilizes.
\end{defn}

Now, our category is ``nice'' if it is Krull-Schmidt with local d.c.c.\ on ideals. In fact, this setting is very common in representation theory. Assume that $k$ is a commutative artinian ring and $\C$ is a skeletally small $k$-category ($\Hom$-spaces are $k$-modules and composition is $k$-linear) and satisfies the following conditions:
\begin{enumerate}
\item[(C1)] $\C$ has splitting idempotents (that is, idempotent morphisms have kernels in $\C$);
\item[(C2)] $\C$ is $\Hom$-finite (that is, $\Hom_\C(X,Y)$ is a finitely generated $k$-module for any objects $X,Y \in \C$).
\end{enumerate}
Then $\C$ is ``nice'':

\begin{lem} \label{lem:nice}
Let $k$ be a commutative artinian ring and $\C$ be a skeletally small $\Hom$-finite $k$-category with splitting idempotents. Then $\C$ is Krull-Schmidt with local d.c.c.\ on ideals.
\end{lem}

\begin{proof}
It is a well known fact that $\C$ is Krull-Schmidt under the assumption. It is straightforward to show that $\mathfrak I(X,Y)$ is a $k$-submodule of $\Hom_\C(X,Y)$ for any ideal $\mathfrak I$ and any pair of objects $X,Y \in \C$. Hence $\C$ has clearly local d.c.c.\ on ideals thanks to (C2).
\end{proof}

\noindent
As a consequence, we can give plenty of examples of ``nice'' categories:
\begin{itemize}
\item $\modL$ for an artin algebra $\Lambda$;
\item $D^b(\Lambda)$, the derived bounded category for an artin algebra $\Lambda$;
\item The category of finite dimensional modules over any algebra over a field;
\end{itemize}
and many others.

Let us start with the proof of the aforementioned statement. First we need a technical lemma.

\begin{lem} \label{lem:limit}
Let $\C$ be a Krull-Schmidt category with local d.c.c.\ on ideals. Let $X,Y \in \C$ and $\alpha$ be a limit ordinal. Then there is $\beta < \alpha$ \st $\rad_\C^\beta(X,Y) = \rad_\C^\alpha(X,Y)$.
\end{lem}

\begin{proof}
Since $\C$ has local d.c.c.\ on ideals, the decreasing chain $(\rad_\C^\gamma(X,Y))_{\gamma<\alpha}$ is stationary. Therefore, there is $\beta<\alpha$ \st
$$ \rad_\C^\beta(X,Y) = \bigcap_{\gamma<\alpha} \rad_\C^\gamma(X,Y) = \rad_\C^\alpha(X,Y). $$
\end{proof}

Now, we are in a position to give the structure theorem for idempotent ideals:

\begin{thm} \label{thm:idemp}
Let $\C$ be a Krull-Schmidt category with local d.c.c.\ on ideals. Let $\mathfrak I$ be an idempotent ideal of $\C$ and $f \in \mathfrak I$. Then there are $f_1, f_2 \in \mathfrak I$ \st $f = f_1 + f_2$, the morphism $f_1$ is generated by identity morphisms from $\mathfrak I$, and $f_2 \in \rad_\C^*$.
\end{thm}

\begin{proof}
We will prove the following statement for all ordinal numbers $\alpha$ by induction:

\medskip
\begin{tabular}{cp{28em}}
$(*)$: &
For every $f \in \mathfrak I$ there are $f_{\alpha,1}, f_{\alpha,2} \in \mathfrak I$ \st $f = f_{\alpha,1} + f_{\alpha,2}$, the morphism $f_{\alpha,1}$ is generated by identity morphisms from $\mathfrak I$, and $f_{\alpha,2} \in \rad_\C^\alpha$.
\end{tabular}
\medskip

Then the theorem will follow if we take $\alpha$ sufficiently big. Let $f: X \to Y$ be a morphism from $\mathfrak I$---we can without loss of generality assume that $X$ and $Y$ are indecomposable.

For $\alpha = 0$, we can simply take $f_{0,1} = 0$ and $f_{0,2} = f$. If $\alpha$ is non-zero finite, we can construct by induction morphisms $g^1,g^2,\dots,g^\alpha \in \mathfrak I$ \st $f = g^1 g^2 \dots g^\alpha$. The morphisms $g^i$, $1 \leq i \leq \alpha$, are not necessarily morphisms between indecomposable objects of $\C$, but we can write $f$ as a finite sum of compositions of morphisms between indecomposables. That is:
$$ f = \sum_j g^{1j} g^{2j} \dots g^{\alpha j}, $$
where we take $g^{ij}$ as components of $g^i$, so that all $g^{ij}$ are in $\mathfrak I$. Finally, we can take $f_{\alpha,1}$ as the sum of those compositions $g^{1j} g^{2j} \dots g^{\alpha j}$ where at least one of the morphisms in the composition is invertible, and $f_{\alpha,2}$ the sum of the remaining compositions. Then clearly $f_{\alpha,1}$ is generated by identities from $\mathfrak I$ and $f_{\alpha,2} \in \rad_\C^\alpha$.

If $\alpha$ is a limit ordinal, there is an ordinal $\beta < \alpha$ \st $\rad_\C^\beta(X,Y) = \rad_\C^\alpha(X,Y)$ by Lemma~\ref{lem:limit}. Of course, $\beta$ depends on $X$ and $Y$. Hence we can set $f_{\alpha,1} = f_{\beta,1}$ and $f_{\alpha,2} = f_{\beta,2}$, where the existence of $f_{\beta,1}, f_{\beta,2}$ is given by inductive hypothesis.

Assume now that $\alpha$ is an infinite non-limit ordinal and $g_{\beta,1}, g_{\beta,2}$ have been already constructed for all $g \in \mathfrak I$ and $\beta<\alpha$. We can write $\alpha = \beta + n$ where $\beta$ is a limit ordinal and $n \ge 1$ is a natural number. Since $\mathfrak I$ is idempotent, we can as in the finite case construct $g^1,g^2,\dots,g^{n+1} \in \mathfrak I$ \st $f = g^1 g^2 \dots g^{n+1}$. By inductive hypothesis, we can for each $1 \leq i \leq n+1$ write $g^i = g^i_{\beta,1} + g^i_{\beta,2}$ where $g^i_{\beta,1}$ is generated by identity morphisms from $\mathfrak I$ and $g^i_{\beta,2} \in \mathfrak I \cap \rad_\C^\beta$. Now,
$$ f = \sum g^1_{\beta,k_1} g^2_{\beta,k_2} \dots g^{n+1}_{\beta,k_{n+1}} $$
where the sum is running through all tuples $(k_1, k_2, \dots, k_{n+1}) \in \{1,2\}^{n+1}$. Put $f_{\alpha,2} = g^1_{\beta,2} g^2_{\beta,2} \dots g^{n+1}_{\beta,2}$ and $f_{\alpha,1} = f - f_{\alpha,2}$. Then it immediately follows by the choice of $g^i_{\beta,1}$ and $g^i_{\beta,2}$ that $f_{\alpha,1}$ is generated by identity morphisms from $\mathfrak I$ and $f_{\alpha,2} \in (\rad_\C^\beta)^{n+1} = \rad_\C^\alpha$.
\end{proof}

Just by reformulating Theorem~\ref{thm:idemp}, we get the following corollary:

\begin{cor} \label{cor:idemp}
Let $\C$ be a Krull-Schmidt category with local d.c.c.\ on ideals. Let $\mathfrak I$ be an idempotent ideal of $\C$, $\mathfrak L$ be a representative set of identity maps contained in $\mathfrak I$, and let $\mathfrak R = \mathfrak I \cap \rad_\C^*$. Then $\mathfrak I$ is generated, as an ideal of $\C$, by $\mathfrak L \cup \mathfrak R$.
\end{cor}

By combining the above statements, we can also characterize the situation when ideals are idempotent exactly when they are generated by a set of identity maps.

\begin{cor} \label{cor:idemp_gen}
Let $\C$ be a Krull-Schmidt category with local d.c.c.\ on ideals. Then the following are equivalent:
\begin{enumerate}
\item Every idempotent ideal of $\C$ is generated by a set of identity maps.
\item $\rad_\C^* = 0$.
\end{enumerate}
\end{cor}

\begin{proof}
$(1) \implies (2)$. If $\rad_\C^* \ne 0$, then by Lemma~\ref{lem:rad_tr_char} it is a non-zero idempotent ideal without identity maps, hence $(1)$ does not hold.

$(2) \implies (1)$. This is immediate by Corollary~\ref{cor:idemp} since, assuming $(2)$, we always get $\mathfrak R = 0$.
\end{proof}

\section{Telescope conjecture for module categories}
\label{sec:tcmc_mod}

The aim of this section is to prove TCMC for algebras with vanishing transfinite radicals. First, we need to collect some general results about TCMC from~\cite{SS}. Even though the results are often proved under weaker assumptions and work almost unchanged for left coherent rings, we specialize them to artin algebras since this is our main concern here.

\begin{prop} \label{prop:cnt_tcmc} \cite[Theorems 3.5, 4.8 and 4.9]{SS}
Let $\Lambda$ be an artin algebra, $(\A,\B)$ be a hereditary cotorsion pair in $\ModL$ \st $\B$ is closed under unions of well ordered chains, and $\mathfrak I$ be the ideal of all morphisms in $\modL$ which factor through some (infinitely generated) module from $\A$. Then:
\begin{enumerate}
\item $(\A,\B)$ is of countable type.
\item $\B = \Ker \Ext^1_\Lambda(\mathfrak I,-) = \{X \in \ModL \mid \Ext^1(f,X)=0 \;\; (\forall f \in \mathfrak I) \}$.
\item Every countably generated module in $\A$ is the direct limit of a countable chain
$$ C_1 \overset{f_1}\to C_2 \overset{f_2}\to C_3 \overset{f_3}\to \dots $$
of finitely generated modules \st $f_i \in \mathfrak I$ for each $i \ge 1$.
\end{enumerate}
\end{prop}

We also need a technical lemma about filtrations which has been studied in~\cite{FL,SaT,ST}, and whose origins can be traced back to an ingenious idea of Paul Hill. Let us recall definitions.

\begin{defn} \label{defn:filtr}
Given a class of modules $\clS$, an \emph{$\clS$-filtration} of a module $M$ is a well-ordered chain $(M_\alpha \mid \alpha\leq\sigma)$ of submodules of $M$ \st $M_0=0$, $M_\sigma=M$, $M_\alpha = \bigcup_{\beta<\alpha} M_\beta$ for each limit ordinal $\alpha\leq\sigma$, and $M_{\alpha+1}/M_\alpha$ is isomorphic to a module from $\clS$ for each $\alpha<\sigma$. A module is called \emph{$\clS$-filtered} if it possesses (at least one) $\clS$-filtration.
\end{defn}

We will use the following specializations of a general statement from~\cite{ST} for finitely and for countably presented modules:

\begin{lem} \cite[Theorem 6]{ST}. \label{lem:hill}
Let $\mathcal S$ be a set of finitely (countably, resp.) presented modules over an arbitrary ring and $M$ be a module possessing an $\mathcal S$-filtration $(M_\alpha \mid \alpha\leq\sigma)$. Then there is a family $\mathcal F$ of submodules of $M$ such that:
\begin{enumerate}
\item $M_\alpha \in \mathcal F$ for all $\alpha\leq\sigma$.
\item $\mathcal F$ is closed under arbitrary sums and intersections.
\item For each $N,P \in \mathcal F$ \st $N \subseteq P$, the module $P/N$ is $\mathcal S$-filtered.
\item For each $N \in \mathcal F$ and a finite (countable, resp.) subset $X \subseteq M$, there is $P \in \mathcal F$ \st $N \cup X \subseteq P$ and $P/N$ is finitely (countably, resp.) presented.
\end{enumerate}
\end{lem}

Most of what we need to do now before proving the main results is to observe that the ideal $\mathfrak I$ from Proposition~\ref{prop:cnt_tcmc} is always idempotent. We state this statement for artin algebras, but it again admits an almost verbatim generalization to left coherent rings.

\begin{lem} \label{lem:tcmc_idemp}
Let $\Lambda$, $(\A,\B)$ and $\mathfrak I$ be as in Proposition~\ref{prop:cnt_tcmc}. Then $\mathfrak I$ is an idempotent ideal of $\modL$.
\end{lem}

\begin{proof}
Let $f: X \to Y$ be a morphism from $\mathfrak I$. By definition, $f$ factors as $X \overset{g}\to A \overset{h}\to Z$ for some $A \in \A$. Since $(\A,\B)$ is of countable type, $A$ must be filtered by countably generated modules from $\A$ \cite[Theorem 10]{ST}. By Lemma~\ref{lem:hill}, we can find a countably generated submodule $A' \subseteq A$ \st $\Img g \subseteq A'$ and $A' \in \A$. More precisely, we use part (4) of the countable version of Lemma~\ref{lem:hill} for $N=0$ and $X$ a finite set of generators of $\Img g$. Hence, $f$ factors as $X \overset{g'}\to A' \overset{h'}\to Z$, and, by Proposition~\ref{prop:cnt_tcmc}, we can express $A'$ as the direct limit of a system
$$ C_1 \overset{f_1}\to C_2 \overset{f_2}\to C_3 \overset{f_3}\to \dots $$
of finitely generated modules \st $f_i \in \mathfrak I$ for each $i \ge 1$. Finally, since $X$ is finitely generated, $g'$ factors through $C_i$ for some $i \ge 1$. But then we can write $f = h'  v f_{i+1} f_i u$ for some morphisms $u$ and $v$, and clearly both $f_i u$ and $h'  v f_{i+1}$ are in $\mathfrak I$. Hence $f \in \mathfrak I^2$ and $\mathfrak I$ is idempotent.
\end{proof}

Now, we can equivalently rephrase Conjecture (B) in the language of ideals:

\begin{prop} \label{prop:tcmc_idemp}
Let $\Lambda$, $(\A,\B)$ and $\mathfrak I$ be as in Proposition~\ref{prop:cnt_tcmc}. Then the following are equivalent:
\begin{enumerate}
 \item $(\A,\B)$ is of finite type.
 \item $\mathfrak I$ is generated by a set of identity morphisms from $\modL$.
\end{enumerate}
\end{prop}

\begin{proof}
$(1) \implies (2)$. Assume that $(\A,\B)$ is of finite type, that is, $\B = \Ker\Ext^1_\Lambda(\clS,-)$ for some set $\clS$ of finitely generated modules. We can without loss of generality assume that $\clS$ is a representative set of all finitely generated modules in $\A$.

We claim that $\mathfrak I$ is then generated by the set $\{ 1_X \mid X \in \clS \}$. To this end we recall that under our assumption, $\A$ consists precisely of direct summands of $\clS$-filtered modules (see~\cite[Theorem 2.2]{T} or~\cite[Corollary 3.2.3]{GT}). Hence, if $f: X \to Y$ is a morphism from $\mathfrak I$, then it factors as $X \overset{g}\to A \overset{h}\to Z$ for some $\clS$-filtered module $A$. Using part (4) of the finite version of Lemma~\ref{lem:hill} for $N=0$ and a finite set $X$ of generators of $\Img g$, we can find a module $A' \subseteq A$ \st $A'$ is isomorphic to some module in $X \in \clS$ and $\Img g \subseteq A'$. Thus, $f$ factors through $1_X$ and since $f$ was chosen arbitrarily, the claim is proved.

$(2) \implies (1)$. Suppose that $\clS$ is a set of finitely generated modules \st $\{1_X \mid X \in \clS\}$ generates $\mathfrak I$. It is straightforward by Proposition~\ref{prop:cnt_tcmc}~(2) that $\B = \bigcap_{X \in \clS} \Ker \Ext^1_\Lambda(1_X,-)$. But this is exactly the same as saying that $\B = \Ker \Ext^1_\Lambda(\clS,-)$. Hence, the cotorsion pair $(\A,\B)$ is of finite type.

\end{proof}

Finally, we can prove TCMC formulated as Conjecture (B) for those artin algebras $\Lambda$ for which $\rad_\Lambda^*=0$. Note that all what we need to do in view of Lemma~\ref{lem:tcmc_idemp} and Proposition~\ref{prop:tcmc_idemp} is to show that certain idempotent ideals are generated by identities, and this is always the case when $\rad_\Lambda^*=0$. As mentioned above, $\rad_\Lambda^*=0$ whenever $\Lambda$ is a domestic standard selfinjective algebra~\cite{KeS} or a domestic special biserial algebra~\cite{S} over an algebraically closed field.

\begin{thm} \label{thm:tcmc_rad0}
Let $\Lambda$ be an artin algebra \st $\rad_\Lambda^* = 0$. Then every hereditary cotorsion pair $(\A,\B)$ in $\ModL$ \st $\B$ is closed under unions of well ordered chains is of finite type.
\end{thm}

\begin{proof}
Let $\mathfrak I$ be the ideal of all morphisms in $\modL$ which factor through some module from $\A$. Then $\mathfrak I$ is an idempotent ideal by Lemma~\ref{lem:tcmc_idemp} and, therefore, generated by a set of identity maps by Corollary~\ref{cor:idemp_gen}. The latter is equivalent to saying that $(\A,\B)$ is of finite type by Proposition~\ref{prop:tcmc_idemp}.
\end{proof}

Another condition on an artin algebra $\Lambda$ which seems to be closely related to vanishing of the transfinite radical and the domestic representation type is that of the Krull-Gabriel dimension of $\Lambda$ being an ordinal number. Let us recall first that the category $\C(\Lambda) = \modmodL$ of finitely presented covariant additive functors $\modL \to \Ab$ is an abelian category, and we can inductively define a filtration
$$
\clS_0 \subseteq \clS_1 \subseteq \clS_2 \subseteq \dots
\subseteq \clS_\alpha \subseteq \clS_{\alpha+1} \subseteq \dots
$$
of Serre subcategories of $\C(\Lambda)$ as follows: Let $\clS_0$ be the full subcateory of $\C(\Lambda)$ formed by functors of finite length, and for each ordinal number $\alpha$, let $\clS_{\alpha+1}$ be the full subcategory of all functors whose image under the localization functor $\C(\Lambda) \to \C(\Lambda)/\clS_\alpha$ is of finite length. At limit ordinals $\alpha$, we take just the unions $\clS_\beta = \bigcup_{\beta<\alpha} \clS_\alpha$. We refer to~\cite[\S7]{K3} for more details and further references. The construction leads to the following definition:

\begin{defn} \label{defn:kgdim}
The \emph{Krull-Gabriel dimension} of an artin algebra $\Lambda$ is defined as $\KGdim \Lambda = \alpha$ where $\alpha$ is the least ordinal number \st $\clS_\alpha = \C(\Lambda)$. If no such $\alpha$ exists, one puts $\KGdim \Lambda = \infty$.
\end{defn}

As a consequence of a deeper and more refined theorem, \cite[Corollary 8.14]{K3} shows that $\radL^* = 0$ whenever $\KGdim\Lambda < \infty$. In particular, we get as a corollary of Theorem~\ref{thm:tcmc_rad0} that TCMC holds for any artin algebra with ordinal Krull-Gabriel dimension:

\begin{cor} \label{cor:tcmc_kgdim}
Let $\Lambda$ be an artin algebra \st $\KGdim \Lambda < \infty$. Then every hereditary cotorsion pair $(\A,\B)$ in $\ModL$ \st $\B$ is closed under unions of well ordered chains is of finite type.
\end{cor}

\begin{rem}
The concept of the Krull-Gabriel dimension has been nicely illustrated by Geigle for tame hereditary algebras $\Lambda$ in~\cite{Gei}, where he explicitly computed that $\KGdim \Lambda = 2$ and described the localization categories $\clS_1/\clS_0$ and $\clS_2/\clS_1$.

The proof of the fact that $\KGdim \Lambda < \infty$ implies $\radL^* = 0$ in~\cite{K3} goes through a stronger statement and involves many technical arguments. There is, however, a more elementary way to see this. Namely, one can define a so called m-dimension of a modular lattice following~\cite[\S 10.2]{P2}. Then $\KGdim \Lambda$ is equal to the m-dimension of the lattice of subobjects in $\modmodL$ of the forgetful functor $\Hom_\Lambda(\Lambda,-)$, \cite[7.2]{K3}. Such subobjects precisely correspond to pairs $(M,m)$ where $M \in \modL$ and $m \in M$, and $(M',m')$ corresponds to a subobject of $(M,m)$ \iff there is a homomorphism $f: M \to M'$ in $\modL$ \st $f(m) = m'$, \cite[7.1]{K3}. Now, $\KGdim \Lambda = \infty$ \iff there is a factorizable system in $\modL$ in the sense of~\cite{P}. Existence of such a factorizable system is easily implied by Lemma~\ref{lem:rad_tr_maps} or \cite[Proposition 0.6]{P} if $\radL^* \ne 0$.

The Krull-Gabriel dimension of $\Lambda$ gives also a strong link to model theory of modules, as it is equal to the m-dimension of the lattice of primitive positive formulas in the first order theory of $\Lambda$-modules. We refer to~\cite[Proposition 0.3]{P} and~\cite[\S 12]{P2} for more details.
\end{rem}

\section{Telescope conjecture for triangulated categories}
\label{sec:tcmc_triang}

We also shortly recall the application on the telescope conjecture for triangulated categories. If $\Lambda$ is a selfinjective artin algebra, then the stable module category $\stModL$ modulo injective modules is \emph{triangulated} in the sense of~\cite[IV]{GM} or \cite[I]{H}. The triangles are, up to isomorphism, of the form
$$ X \overset{f}\to Y \overset{g}\to Z \overset{h}\to \Sigma X $$
where $0 \to X \overset{f}\to Y \overset{g}\to Z \to 0$ is a short exact sequence in $\ModL$, and the suspension functor $\Sigma: \stModL \to \stModL$ corresponds to taking cosyzygies in $\ModL$. Clearly, $\Sigma$ is an auto-equivalence of $\stModL$ and the corresponding inverse $\Sigma^{-1}$ is given by taking syzygies in $\ModL$.

An object $X$ in a triangulated category with (set-indexed) coproducts is called \emph{compact} if the representable functor $\Hom(X,-)$ commutes with coproducts. In particular, an object $X \in \stModL$ is compact \iff it is isomorphic to a finitely generated $\Lambda$-module in $\stModL$ (see~\cite[\S 1.5]{K1} or~\cite[\S 6.5]{K4}).

A full triangulated subcategory $\X$ of $\stModL$ is called \emph{localizing} if it is closed under forming coproducts in $\stModL$. A localizing subcategory $\X$ is called \emph{smashing} if the inclusion $\X \hookrightarrow \stModL$ has a right adjoint which preserves coproducts. We say that a localizing subcategory $\X$ is \emph{generated} by a class $\C$ of objects if there is no proper localizing subclass of $\X'$ of $\X$ \st $\C \subseteq \X'$. We refer to~\cite{K1,K2} for a thorough discussion of these concepts. It follows that $\stModL$ is a \emph{compactly generated} triangulated category, that is, $\stModL$ is generated, as a localizing class, by a set of compact objects.

The telescope conjecture studied in~\cite{K1,K2} asserts that every smashing localizing subcategory of a compactly generated triangulated category is generated by a set of compact objects. Even though it is generally false as mentioned in the introduction, we can give an affirmative answer in a special case. Namely Theorem~\ref{thm:tcmc_rad0} together with results from~\cite{KrS} imply that the conjecture holds for $\stModL$ where $\Lambda$ is a selfinjective artin algebra with vanishing transfinite radical.

\begin{thm} \label{thm:tcmc2_rad0}
Let $\Lambda$ be a selfinjective artin algebra \st $\rad_\Lambda^* = 0$. Let $\X$ be a smashing localizing subcategory of $\stModL$. Then $\X$ is generated by a set of finitely generated $\Lambda$-modules.
\end{thm}

\begin{proof}
We know that Conjecture (B) (see page~\pageref{conj:B}) holds for $\Lambda$ by Theorem~\ref{thm:tcmc_rad0}. Hence also Conjecture (A) holds by the discussion in Section~\ref{sec:prelim}. The rest follows immediately from~\cite[Corollary 7.7]{KrS}.
\end{proof}

\section{Examples}
\label{sec:expl}

We conclude with some examples of particular representation-infinite selfinjective algebras with vanishing transfinite radical.

\begin{expl}
The simplest example is probably the exterior algebra of a $2$-di\-men\-si\-o\-nal vector space over an algebraically closed field. That is, $\Lambda_2 = k\spanned{x,y}/(x^2,y^2,xy+yx)$. It is a special biserial algebra in the sense of~\cite{SW} and it has, up to rotation equivalence and inverse, only one band $xy^{-1}$. In particular, $\Lambda_2$ is domestic and we have exactly one one-parametric family of indecomposable modules in each even dimension. For example, we have $M_{(a:b)} = \Lambda_2/\Lambda_2 (ax + by)$ for each $(a\!\!:\!\!b) \in \mathbb{P}^1(k)$ in dimension $2$. Thus, $\rad_{\Lambda_2}^* = 0$ by~\cite[Theorem 2]{S}.

With a little more effort, we can classify all smashing localizations and all hereditary cotorsion pairs with the right hand class closed under unions of chains. Using the representation theory of special biserial algeras, one can readily compute the Auslander-Reiten quiver of $\Lambda_2$. It consists of a family $(\T_{(a:b)} \mid (a\!\!:\!\!b) \in \mathbb{P}^1(k))$ of homogeneous tubes, the corresponding quasi-simples being precisely the modules $M_{(a:b)}$ above. In addition, there is one more component, which we denote by $\C$, of the form
$$
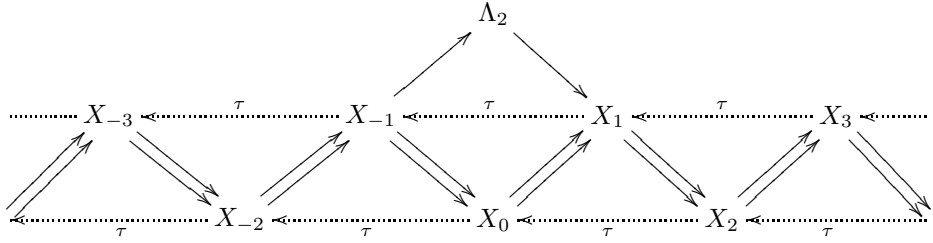
\xymatrix{
&&&& \Lambda_2 \ar[dr]
\\
\ar@{.}[r] &
X_{-3} \ar@<0.5ex>[dr] \ar@<-0.5ex>[dr] &
&
X_{-1} \ar@<0.5ex>[dr] \ar@<-0.5ex>[dr] \ar[ur] \ar@{.>}[ll]_\tau &
&
X_1 \ar@<0.5ex>[dr] \ar@<-0.5ex>[dr] \ar@{.>}[ll]_\tau &
&
X_3 \ar@<0.5ex>[dr] \ar@<-0.5ex>[dr] \ar@{.>}[ll]_\tau &
\ar@{.>}[l]
\\
\ar@<0.5ex>[ur] \ar@<-0.5ex>[ur] &
&
X_{-2} \ar@<0.5ex>[ur] \ar@<-0.5ex>[ur] \ar@{.>}[ll]^\tau &
&
X_0 \ar@<0.5ex>[ur] \ar@<-0.5ex>[ur] \ar@{.>}[ll]^\tau &
&
X_2 \ar@<0.5ex>[ur] \ar@<-0.5ex>[ur] \ar@{.>}[ll]^\tau &
&
\ar@{.>}[ll]^\tau
}
$$
where $X_0$ is the unique simple module, and $X_n$ and $X_{-n}$ are the string modules corresponding to the strings $(yx^{-1})^n$ and $(x^{-1}y)^n$, respectively. In particular, $\dim_k X_n = 2\cdot|n| + 1$. It is easy to compute that $\Omega^-(X_n) \cong X_{n+1}$ and $\Omega^-(M) = M$ for each indecomposable finite dimensional module in a tube. This describes the restriction of the suspension functor $\Sigma: \stModl{\Lambda_2} \to \stModl{\Lambda_2}$ to $\stmodl{\Lambda_2}$.

We recall that a full triangulated subcategory $\X_0$ of $\stmodl{\Lambda_2}$ is called \emph{thick} if it is closed under direct summands. There is a bijective correspondence between thick subcategories $\X_0$ of $\stmodl{\Lambda_2}$ and localizing subcategories $\X$ of $\stModl{\Lambda_2}$ generated by a set of compact objects. More precisely, if $\X$ is generated by $\X_0 \subseteq \stmodl{\Lambda_2}$ and $\X_0$ is thick, then $\X \cap \stmodl{\Lambda_2} = \X_0$, \cite[2.2]{N1}. It is clear that each thick subcategory is uniquely determined by its indecomposable objects.

\begin{sloppypar}
We will now describe thick subcategories of $\stmodl{\Lambda_2}$. It is straightforward to check that if an indecomposable non-injective module $M \in \modl{\Lambda_2}$ is contained in a thick subcategory $\X_0$, then all modules in the same component of the Auslander-Reiten quiver are in $\X_0$, too. On the other hand, if $\T_p$ is a tube for some $p \in \mathbb{P}^1(k)$, then one can check that in $\modl{\Lambda_2}$, the additive closure of $\T_p \cup \{\Lambda_2\}$ equals to
$$
\{ X \in \modl{\Lambda_2} \mid
\underline{\Hom}_{\Lambda_2}(X,\T_q) = 0
= \underline{\Hom}_{\Lambda_2}(\T_q,X)
\;\; (\forall q \in \mathbb{P}^1(k) \setminus \{p\})
\}
$$
Therefore, $\add (\T_p \cup \{\Lambda_2\})$ is closed under extensions, syzygies and cosyzygies in $\modl{\Lambda_2}$, and consequently $\add \T_p$ is thick in $\stmodl{\Lambda_2}$. It is easy to see that $\underline{\Hom}_{\Lambda_2}(\T_p,\T_q) = 0$ for $p \ne q$, so the additive closure of any set of tubes is thick in $\stmodl{\Lambda_2}$. Finally, there is an exact sequence $0 \to M \to X_m \to X_{m+1} \to 0$ for each $m < 0$ and each quasi-simple module $M$ in a tube; hence a thick subcategory containing the component $\C$ contains all the tubes, too. When summarizing all the facts (and using Theorem~\ref{thm:tcmc2_rad0}), we obtain the following classification:
\end{sloppypar}

\begin{prop}
Let $k$ be an algebraically closed field, $\Lambda_2 = k\spanned{x,y}/(x^2,y^2,xy+yx)$, and $\C$ and $\T_p$, $p \in \mathbb{P}^1(k)$, be the components of the Auslander-Reiten quiver of $\Lambda_2$ as above. Then each smashing localizing class $\X$ in $\stModl{\Lambda_2}$ is generated by $\X_0 = \X \cap \stmodl{\Lambda_2}$, and the possible intersections $\X_0$ are classified as follows:
\begin{enumerate}
\item $\X_0 = 0$; or
\item $\X_0$ is the additive closure of $\bigcup_{p \in P} \T_p$ for some $P \subseteq \mathbb{P}^1(k)$; or
\item $\X_0 = \stmodl{\Lambda_2}$.
\end{enumerate}
\end{prop}

In the same spirit, we can classify the hereditary cotorsion pairs $(\A,\B)$ in $\Modl{\Lambda_2}$ such that $\B$ is closed under unions of chains. Recall that a subcategory $\A_0$ of $\modl{\Lambda_2}$ is called \emph{resolving} if it contains $\Lambda_2$ and it is closed under extensions, kernels of epimorphisms and direct summands. There is a bijective correspondence between resolving subcategories $\A_0$ in $\modl{\Lambda_2}$ and hereditary cotorsion pairs $(\A,\B)$ of finite type in $\Modl{\Lambda_2}$, \cite[2.5]{AT}. Note that if $\A_0$ is resolving and contains a module $X_m \in \C$, it must contain all $X_z$, $z\le m$, and all tubes. On the other hand, it is not difficult to see that there is an exact sequence $0 \to X_n \to U \to X_{-k} \to 0$ with an indecomposable (string) module $U$ from a tube for each $n,k > 0$. Hence $\A_0$ must contain all of $\C$, too. We will leave details of the following statement (using Theorem~\ref{thm:tcmc_rad0}) for the reader:

\begin{prop}
\begin{sloppypar}
Let $k$ be an algebraically closed field, $\Lambda_2 = k\spanned{x,y}/(x^2,y^2,xy+yx)$, and $\C$ and $\T_p$, $p \in \mathbb{P}^1(k)$, be the components of the Auslander-Reiten quiver of $\Lambda_2$ as above. Let $(\A,\B)$ be a hereditary cotorsion pair in $\Modl{\Lambda_2}$ \st $\B$ is closed under unions of chains, and let $\A_0 = \A \cap \modl{\Lambda_2}$. Then $\B = \Ker \Ext^1_{\Lambda_2}(\A_0,-)$, and the possible classes $\A_0$ are classified as follows:
\end{sloppypar}
\begin{enumerate}
\item $\A_0 = \add\{\Lambda_2\}$; or
\item $\A_0$ is the additive closure of $\{\Lambda_2\} \cup \bigcup_{p \in P} \T_p$ for $P \subseteq \mathbb{P}^1(k)$; or
\item $\A_0 = \modl{\Lambda_2}$.
\end{enumerate}
\end{prop}
\end{expl}

\begin{expl}
A recipe for construction of more complicated examples is given in~\cite{KeS}. Let $B$ be a representation-infinite tilted algebra of Euclidean type over an algebraically closed field and $\hat B$ be its repetitive algebra. Put $\Lambda = \hat B/G$ where $G$ is an admissible infinite cyclic group of $k$-linear automorphisms of $\hat B$ (see~\cite[\S 1]{Sk2} for unexplained terminology). Then $\Lambda$ is selfinjective and $\radL^* = 0$ by the main result of~\cite{KeS}.

We illustrate the construction on $B = k(\cdot\!\rightrightarrows\!\cdot)$, the Kronecker algebra. The repetitive algebra $\hat B$ is then given by the following infinite quiver with relations:
$$
\xymatrix@1{
\ar@{.}[r] &
\cdot \ar@<0.5ex>[r]^{x_0} \ar@<-0.5ex>[r]_{y_0} &
\cdot \ar@<0.5ex>[r]^{x_1} \ar@<-0.5ex>[r]_{y_1} &
\cdot \ar@<0.5ex>[r]^{x_2} \ar@<-0.5ex>[r]_{y_2} &
\cdot \ar@<0.5ex>[r]^{x_3} \ar@<-0.5ex>[r]_{y_3} &
\cdot \ar@{.}[r] &
}
$$
$$
x_{i+1} x_i - y_{i+1} y_i = 0, \quad x_{i+1} y_i = 0, \quad y_{i+1} x_i = 0
\qquad \textrm{for each $i \in \Z$.}
$$

Let $n \ge 1$ and $\bar q = (q_1, \dots, q_n)$ be an $n$-tuple of non-zero elements of $k$. It is not difficult to see that we get the algebra $\Lambda_{n,\bar q}$ described by the quiver and relations below as $\hat B/G$ for a suitable $G$:
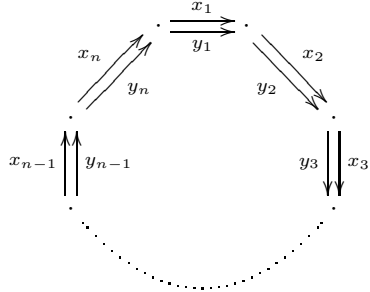
$$
\xymatrix{
&
\cdot \ar@<0.5ex>[r]^{x_1} \ar@<-0.5ex>[r]_{y_1} &
\cdot \ar@<0.5ex>[rd]^{x_2} \ar@<-0.5ex>[rd]_{y_2} & \\
\cdot \ar@<0.5ex>[ur]^{x_n} \ar@<-0.5ex>[ur]_{y_n} &&&
\cdot \ar@<0.5ex>[d]^{x_3} \ar@<-0.5ex>[d]_{y_3} \\
\cdot \ar@<0.5ex>[u]^{x_{n-1}} \ar@<-0.5ex>[u]_{y_{n-1}}  \ar@{.}@/_2.5pc/[rrr] &&&
\cdot
\\
\\
}
$$
$$
x_{i+1} y_i + q_i y_{i+1} x_i = 0, \;\; x_{i+1} x_i = 0, \;\; y_{i+1} y_i = 0
\quad \textrm{for each $i \in \{1,2,\dots,n\}$.}
$$
The addition in indicies of arrows above is considered modulo $n$. It is easy to see that $\Lambda_{n,\bar q}$ is special biserial and there are exactly $n$ one-parametric families of indecomposable $\Lambda_{n,\bar q}$-modules in each even dimension. They correspond to the bands $x_i y_i^{-1}$. In fact, if $n=1$ and $q_1 = 1$, we get precisely the exterior algebra on a $2$-dimensional space.
\end{expl}


\end{document}